%% file: rigidityZ2eigensection.tex
\documentclass[11pt]{amsart}
\usepackage{amsmath, amssymb, amscd, mathrsfs, slashed, enumerate, url}
\usepackage{color}
\usepackage{extsizes}
\usepackage{xcolor}
\usepackage[all,cmtip]{xy}
\usepackage{multicol}
\usepackage{indentfirst}
\usepackage{latexsym}
\usepackage{bm}
\usepackage{graphicx}
\usepackage{subfigure,esint}
\usepackage{float}
\usepackage{verbatim}
\usepackage{comment}
\usepackage[top=1in, bottom=1in, left=1.25in, right=1.25in]{geometry}
\usepackage{epsfig,dsfont,amsthm,amsfonts,amsbsy}
\usepackage[colorlinks]{hyperref}
\usepackage[toc,page]{appendix}
\usepackage{tikz-cd}
\usepackage{amsthm,thmtools,xcolor}
\newtheorem{theorem}{Theorem}[section]

\newtheorem{corollary}[theorem]{Corollary}

\newtheorem{definition}[theorem]{Definition}

\newtheorem{proposition}[theorem]{Proposition}
\newtheorem*{remark}{Remark}

\input{convention.tex}

\usepackage{calligra}
\DeclareMathAlphabet{\mathcalligra}{T1}{calligra}{m}{n}

\declaretheoremstyle[
headfont=\color{blue}\normalfont\bfseries,
bodyfont=\color{blue}\normalfont\itshape,
]{colored}

\usepackage[utf8]{inputenc}
\usepackage[english]{babel}
\usepackage{fancyhdr}
\usepackage{accents}

\usepackage[textsize=small]{todonotes}
\newcommand{\Comment}[2][\empty]{\ifthenelse{\equal{#1}{\empty}}{\todo[color=gray!10]{#2}\ }{\todo[color=gray!10,#1]{#2}}}%inline comment
\title[\normalsize Deformation rigidity for $\ZT$ eigensections]{\Large Deformation rigidity for $\ZT$ eigensections}

\author[\normalsize A. Haydys, S. He, A. Salm]
{
	\begin{minipage}{0.3\textwidth}
		\centering
		\normalsize Andriy Haydys \\
		{ Universit\'e libre de Bruxelles}
	\end{minipage}
	\begin{minipage}{0.36\textwidth}
		\centering
		\normalsize Siqi He\\
		{Academy of Mathematics \\and Systems Science}
	\end{minipage}
	\begin{minipage}{0.3\textwidth}
		\centering
		\normalsize Willem Andriaan Salm\\
		{Universit\'e libre de Bruxelles }
	\end{minipage}
}

\usepackage{etoolbox}

\makeatletter
\let\MakeUppercase\relax % this disables uppercasing authors
\makeatother
	
\date{\today}

\begin{document}

\begin{abstract}
	We prove a rigidity result for certain critical $\mathbb Z/2$ eigensections of the Laplacian on $S^2$ associated to a flat real line bundle determined by a branch–point configuration.  
	More precisely, we show that every minimal non-degenerate critical eigensection is deformation rigid: any sufficiently small deformation of the configuration that still admits a critical eigensection must come from an $\mathrm{SO}(3)$-rotation.  
	This generalizes the rigidity phenomenon previously discovered in symmetric examples of Taubes--Wu.  
\end{abstract}

\begingroup
\def\uppercasenonmath#1{} % this disables uppercasing title
\let\MakeUppercase\relax % this disables uppercasing authors
\maketitle
\endgroup

\section{Introduction}

Recently, various types of $\mathbb Z/2$ harmonic sections attracted a noticeable attention due to their appearance in a number of seemingly unrelated geometric problems. For example, sequences of degenerating flat $\mathrm{SL}_2(\mathbb C)$ connections and diverging sequences of solutions of the generalized Seiberg--Witten equations
\cite{haydyswalpuski2015compactness,taubes2013compactness,taubes2013psl2c,walpuskizhang2019compactness} are related to $\mathbb Z/2$ harmonic forms and spinors. Certain coassociative fibrations of $\mathrm G_2$-manifolds in the adiabatic limit yield $\mathbb Z/2$ harmonic functions~\cite{Donaldson17_AdiabaticLimits}. Infinitesimal deformations of branched special Lagrangian submanifolds are modelled on $\mathbb Z/2$ harmonic forms \cite{HE2023SpecialLagrangian}.
All this provides a strong motivation for studying $\mathbb Z/2$ harmonic sections among which $\mathbb Z/2$ harmonic functions are perhaps the simplest objects to study. 
    
A better understanding of regularity of $\mathbb Z/2$ harmonic functions/forms/spinors is necessary to make a progress in the above mentioned problems such as compactifications of various gauge--theoretic moduli spaces of interest. 
Indeed, the technique currently available allows us to deal effectively with $\Z/2$ harmonic sections whose branching locus is smooth. 
While some progress has been achieved recently in dealing with $\Z/2$ harmonic sections with mildly singular branching loci~\cite{HMT2023IndexGraphs}, our understanding of regularity of branching sets at present is insufficient for intended applications.  
This motivates our studies of certain $\mathbb Z/2$ eigensections of the Laplacian on the 2-sphere since they are models for singularities of $\mathbb Z/2$ harmonic functions on 3-manifolds. 

To explain some details, let $\mbp$ be an unordered set of $2n$ distinct points on $S^2$. 
There is a unique flat real Euclidean line bundle $\cI_\mbp\to S^2\setminus\mbp$ with monodromy $-1$ along each small circle encircling any point in $\mbp$. 
We then obtain the Laplace operator acting on sections of $\cI_\mbp$ in the standard way. Then $u\in \Gamma(\cI_\mbp)$ is said to be an eigensection of the Laplacian (or a $\mathbb Z/2$ eigenfunction) if there exists $\La\in\R$ such that
\begin{equation*}
  \Delta u = \La u\qquad\text{holds on } S^2\setminus\mbp.
\end{equation*}
We are mainly interested in the so called \textit{critical} eigensections, which are solutions $u$ satisfying the following property: near each $p\in\mbp$ we have $|u (x)|\le C\dist (x,p)^{3/2}$.  
We refer to Section~\ref{Sect_Background} for some background material and further details.

Critical eigensections were first introduced by Taubes in his foundational work on the zero sets of 
$\mathbb Z/2$-harmonic spinors \cite{taubes2014zero}, albeit the terminology we use appeared later~\cite{chen2024existence}. In Taubes' work $\mathbb Z/2$ eigensections 
on $S^2$ arise as local models governing the tangent behaviour near singular points.  
This perspective was subsequently developed in a systematic way by Taubes--Wu 
\cite{taubes2020examples,taubes2021topological}, who carried out a detailed analytic and
topological study of the associated eigenvalue problem, established the structure of
critical eigensections, and exhibited explicit families of singularity models on the sphere.  
These results provide a robust framework in which critical $\mathbb Z/2$ eigensections serve as 
the fundamental local models for degenerations of $\Z/2$ harmonic sections.

All critical eigensections necessarily lie in eigenspaces of dimension at least four, 
a consequence of the $\mathrm{SO}(3)$-symmetry of the Laplacian and the multiplicity
phenomena described in \cite{taubes2020examples,taubes2021topological}.  
Motivated by this observation, we call a critical eigensection \emph{minimal} if the 
corresponding eigenspace has dimension exactly four.  
The basic examples of minimal $\mathbb Z/2$ eigensections are the Taubes--Wu
tetrahedral eigensections \cite{taubes2020examples}; it is proved in
\cite{chen2024existence} that these tetrahedrally symmetric solutions are
non-degenerate and deformation rigid.  
Within the class of minimal eigensections, the \emph{non-degenerate} ones---characterised by a natural 
non-vanishing condition on their leading local coefficients---play a particularly 
important role and appear in several geometric and topological contexts
\cite{donaldson2019deformations,HeParker2024Z2Harmonic,Parker2025DeformationsZ2Spinors}.
One can think of critical, non-degenerate, minimal eigensections, which are the main players below, as `generic' among critical eigensections, albeit we do not make an attempt to justify this claim.  

The goal of this article is to understand the deformation theory of such eigensections.
Given a configuration $\mbp$ and a critical, non-degenerate, minimal eigensection $f$, 
we ask whether $f$ persists under small deformations of $\mbp$, or whether it is rigid up 
to the natural $\mathrm{SO}(3)$-action.  
Our main result, whose proof can be found on Page~\pageref{Pf_MainThm} below, shows that the latter holds.

\begin{theorem}
	\label{Thm_DeformRigid}
	A minimal non-degenerate critical $\Z/2$ eigenfunction $f\in \Gamma\big (\MI_{\mathbf p}\big )$ is deformation rigid.
\end{theorem}

To explain, the deformation rigidity means that for any sufficiently small deformation $\tilde\mbp$ of $\mbp$ that still admits
a nearby critical eigensection $\tilde f$ there exists some $R\in \mathrm{SO}(3)$ such that $\tilde\mbp = R\,\mbp$ and $\tilde f = f\circ R$. 
This establishes a general rigidity phenomenon extending the symmetric tetrahedral models 
of Taubes--Wu, and it complements recent work on the existence and structure of critical 
$\mathbb Z/2$ eigenvalues 
\cite{chen2024existence,HaydysSalm2025Z2Sphere} and recent applications \cite{FranceschiniMazzeo2026MinimalSurfaces}.

\medskip
\textbf{Acknowledgements.} 
The authors thank Jiahuang Chen and Rafe Mazzeo for helpful comments.

\section{Background on  \texorpdfstring{$\ZT$}{Z/2} eigensections}
\label{Sect_Background}
In this section we recall the basic setup for $\mathbb Z/2$ eigensections on the sphere, which have been studied and developed in \cite{taubes2014zero,taubes2020examples,taubes2021topological}.
\subsection{\texorpdfstring{$\ZT$}{Z/2} eigensections}
Let $\MC_{2n}$ be the space of unordered configurations of $2n$ distinct points on $(S^2,g_0)$ with $g_0$ being the round metric.
For $\mbp\in\MC_{2n}$, let $\MI_{\mbp}\to S^2\setminus\mbp$ be the unique flat real Euclidean line bundle with monodromy $-1$ around any small loop linking a single point of $\mbp$.
Let
\[
\Delta:\Gamma(\MI_{\mbp})\longrightarrow \Gamma(\MI_{\mbp})
\]
denote the Laplace operator associated with the flat connection on $\MI_{\mbp}$.

Let $\mbH_{\mbp}$ be the completion of $C_c^\infty(S^2\setminus\mbp;\MI_{\mbp})$ with respect to the
energy norm
\[
\|f\|_{\mbH_{\mbp}}^2 \;:=\; \int_{S^2\setminus\mbp} |df|^2 .
\]
A section $f\in \mbH_{\mbp}$ is called a $\ZT$ eigensection at $\mbp$ if
\[
\Delta f \;=\; \lambda f 
\]
for some eigenvalue $\lambda\ge 0$, where the above equality is understood in the sense of distributions on  $S^2\setminus\mbp$. 
For a fixed eigenvalue $\lambda$, denote by 
\[
V_\lambda(\mbp)\;:=\;\{\,f\in \mbH_{\mbp}\;\mid\; \Delta f=\lambda f\,\}
\]
the corresponding eigenspace and write $\mul \lam:=\dim V_\lambda(\mbp)$. 
Notice that we have the orthogonal decomposition $$\mbH_{\mbp}=\oplus_{\lambda}V_{\lambda}(\mbp).$$

\subsection{Local expansions and non-degenerate eigenfunctions}  
Given $p\in\mbp$, let $z=re^{i\theta}$ denote the complex coordinate centred at $p$ obtained from the stereographic projection from $-p$.  
In this coordinate the Laplace operator can be written as
\[
\Delta=\Bigl(\frac{r^2+1}{2}\Bigr)^2
\Bigl(\partial_r^2+\frac1r\,\partial_r+\frac{1}{r^2}\partial_{\theta}^2\Bigr).
\]

Think of sections of $\MI_\mbp$ on a punctured disc $0<|z|<\rho$ as `functions' changing the sign if one goes along a circle centred at the origin just like $\Re  \sqrt z$ does. 
Then any eigensection $f\in\mbH_{\mbp}$ admits a convergent asymptotic expansion
\begin{equation}\label{eq_local_expansion}
	f(z)\sim \Re\!\left(a_p\,z^{1/2}+b_p\,z^{3/2}\right)+\MO(r^{5/2}),
\end{equation}
where $a_p$ and $b_p$ are complex numbers, well defined up to an overall sign.

Let $\pi:\Sigma_{\mbp}\to S^2$ denote the double branched cover with branch locus $\mbp$.  
If $f$ is an eigensection, then its pullback
$\tilde f:=\pi^*f$ is an \emph{odd} eigenfunction of the conical Laplacian of $\pi^*g_{S^2}$, where the pulled-back metric has cone angle $4\pi$ at each branch point.  
Moreover, near a branch point one may choose a local coordinate $w$ on $\Sigma_{\mbp}$ such that the covering map is given by $\pi(w)=w^2=z$.  
A choice of such local coordinates on $\Sigma_{\mbp}$ fixes the signs of the coefficients $a_p$ and $b_p$.  
Throughout the remainder of this article, we make this choice and, hence, consider $a_p$ and $b_p$ as well-defined complex numbers.  
Notice that for any configuration $\tilde\mbp$ sufficiently close to $\mbp$, the corresponding coordinate $\tilde w$ on $\Sigma_{\tilde\mbp}$ is canonically defined.  

It is convenient to choose an order on $\mbp$ so that we can write $\mbp=(p_1,\ldots,p_{2n})$. 
For any finite sum $f\in\mbH_\mbp$ of eigensections we have well-defined  coefficients $a_{p_i}$ and $b_{p_i}$ as in \eqref{eq_local_expansion}.  
We say that the vector
\[
\Tr(f):=(a_{p_1},\ldots,a_{p_{2n}})\in\C^{2n}
\]
encoding the lowest order terms is \emph{the trace} of $f$.

\begin{definition}
	An eigensection $f$ is said to be \emph{critical} if $\Tr(f)=0$.  
	An eigenspace $V_{\lambda}$ (or the eigenvalue $\lambda$) is called \emph{critical} if it contains a critical eigensection.  
	Furthermore, a critical eigensection $f$ is said to be \emph{non-degenerate},  if $b_{p_i}\neq 0$ for all $i$.
\end{definition}

\subsection{Deformations of \texorpdfstring{$\ZT$}{Z/2} eigenvalues}
For a fixed configuration $\mbp$, the spectrum of the Laplacian on $\mbH_\mbp$ is discrete. However, understanding  the dependence of this spectrum on $\mbp$ is crucial for the proof of our main result. 
We begin with the following.

\begin{proposition}{\cite[Proposition 2.5, Proposition 2.6]{taubes2021topological}}
	\label{prop_deformation}
	Fix some $\mathbf{p}\in\mathcal{C}_{2n}$ and let $f\in \mbH_{\mbp}$ be an eigensection of the Laplacian corresponding to an eigenvalue $\lambda(\mathbf{p})$. For $p\in\mbp$, denote by $a_p(f)$ the leading coefficient of $f$ near $p$ as above. 
	%Choosing a complex coordinate $z$ near each $p\in\mathbf{p}$,  $f$ has an asymptotic expansion near $p$ as in~\eqref{eq_local_expansion}. With these choices at hand, the leading coefficients $a_p(f)$ are well-defined up to a sign. 
	
	If $\mathrm{mul}\,\lambda(\mathbf{p}) =1$, then there is a neighbourhood $U$ of $\mbp$ in $\mathcal{C}_{2n}$, such that there is a smooth function $\lambda(-)$ defined on $U$ with the following properties:
	\begin{enumerate}
		\item The value of $\lambda{(-)}$ at $\mathbf{p}$ is $\lambda(\mathbf{p})$;
		\item If $\mathbf{q}$ is contained in $U$, then $\lambda{(\mathbf{q})}$ is an eigenvalue of the Laplacian on $\mbH_\mbq$ with $\mathrm{mul}\,\lambda{(\mathbf{q})}=1$;
		\item If $\mbv$ is a tangent vector of $\mathcal{C}_{2n}$ at $\mathbf{p}$, then the directional derivative of $\lambda{(-)}$ along $\mbv$ is $\frac{\pi}{2}\sum_{p\in \mathbf{p}_f}\Re\,(v(p)a_{p}(f)^2)$, where $f$ denotes an eigensection corresponding to the eigenvalue $\lambda(\mathbf{p})$ with $\|f\|_{L^2}=1$.
	\end{enumerate}
	
	If $N:=\mathrm{mul}\,\lambda(\mathbf{p})>1$, then there exist a neighbourhood $U$ of $\mbp$ and a set of $N$ continuous functions $\mu_{i}(-)(i=1,\cdots,N)$ defined on $U$ with the following properties:
	\begin{enumerate}
		\item $\mu_{i}(\mathbf{p})=\lambda(\mathbf{p})$ for $i=1,\cdots,N$;
		\item If $\mathbf{q}\in U$, then each $\mu_{i}(\mathbf{q})$ is an eigenvalue of the Laplacian on $\mbH_\mbq$;
		\item 
		For each tangent vector $\mbv$ at $\mbp$ consider the symmetric bilinear form on the eigenspace $V_{\lambda(\mathbf{p})}$ defined by 
		\begin{align}
			B_{\mbv}(f,f'):=\frac{\pi}{2}\sum_{p\in\mathbf{p}_f\cap\mathbf{p}_{f'}}\Re\,\big(v(p)a_p(f)a_p(f')\big).\label{TWblf}
		\end{align}
Denote by $\eta_1(\mbv),\cdots,\eta_N(\mbv)$ the eigenvalues of this bilinear form and chose an $L^2$-orthonormal basis $\{f_1,\cdots,f_N\}$ of $V_{\La(\mbp)}$ such that $B_{\mbv}(f_i,f_j)=\eta_i(\mbv)\delta_{ij}$. If $\mathbf{p}(t)$ is a smooth path of configurations such that $\mathbf{p}(0)=\mathbf{p}$ and $\dot\mbp(0)=\mbv$, then for $t$ small, there exist $L^2$ normalized eigensections $f_{i,t}$ with eigenvalues $\mu_i(\mbp(t))$ such that 
		\begin{equation}
			f_{i,t}=f_i+\mathcal{O}(t),\qquad\mu_i(t)=\lambda(\mbp)+t\eta_i(\mbv)+o(t).
		\end{equation}
	\end{enumerate}
\end{proposition}

Consider $S^2$ as a subspace of $\mbR^3$ and denote by $(x_1,x_2,x_3)$ its standard coordinates. 
The vector fields on $S^2$ given by 
\begin{equation}
	\label{eq_rotation}
	\begin{split}
		L_1 = x_2\,\partial_{x_3}-x_3\,\partial_{x_2},\qquad
		L_2 = x_3\,\partial_{x_1}-x_1\,\partial_{x_3},\qquad
		L_3 = x_1\,\partial_{x_2}-x_2\,\partial_{x_1}
	\end{split}
\end{equation}
represent the infinitesimal generators of the $\mathrm{SO}(3)$-action. 
Using the flat structure on $\MI_{\mbp}$, each $L_i$ acts on sections and defines a linear map
\[
L_i:\Gamma(\MI_{\mbp})\to \Gamma(\MI_{\mbp}),
\]
satisfying $[\Delta,L_i]=0$.  
Therefore, if $f\in V_{\lambda}$ is a critical eigensection, then $L_i f$ is also an eigensection corresponding to the same eigenvalue $\lambda$ so that a critical eigenvalue is never simple. 
More precisely, we have the following result:

\begin{proposition}[{\cite[Proposition~4.5, Theorem~1.5]{chen2024existence}, \cite{taubes2021topological}}]
	Let $\mbp\in\MC_{2n}$ with $2n\ge 2$, and let $f\in V_{\lambda}$ be a critical eigensection.  
	Then the eigensections
	\[
	f,\; L_1 f,\; L_2 f,\; L_3 f
	\]
	are linearly independent.  
	In particular, $\mathrm{mul}\,V_{\lambda}\ge 4$.
\end{proposition}

This proposition motivates the following. 

\begin{definition}
	We say that a critical eigensection $f\in V_{\lambda}$ is \emph{minimal} if $\dim V_{\lambda}=4$.
\end{definition}
By \cite[Section 5.3]{chen2024existence}, Taubes--Wu tetrahedral eigensections found in \cite{taubes2020examples} are minimal.

\section{Deformation rigidity}
In this section, we will prove that critical eigensections are deformation-rigid under the hypothesis of Theorem~\ref{Thm_DeformRigid}. 

First observe that the standard action of $\SO(3)$ on $S^2$ yields an action on $\MC_{2n}$. For $R\in\SO(3)$, we therefore obtain a map $\mbH_{\mbp}\to \mbH_{R\cdot\mbp}$,\  $(R\cdot f)(x) = f(R^{-1}x)$, which maps each $V_\lambda(\mbp)$ isomorphically onto $V_\lambda (R\cdot \mbp)$.  

\begin{definition}
	A critical eigensection $f_0\in V_{\lambda_0}(\mbp_0)$ is called \emph{deformation rigid} if there exist a neighbourhood $U\subset \MC_{2n}$ of $\mbp_0$ and a constant $\epsilon$ with the following property: If for some $\mbp\in U$ there exists a critical eigensection $f\in V_{\lambda}(\mbp)$ with $\lambda\in (\lambda_0-\epsilon, \lambda_0+\epsilon)$, then there exists $R\in SO(3)$ such that $R\cdot \mbp=\mbp_0$ and $R\cdot f=f_0$.
\end{definition}

%\begin{definition}
%	A critical eigensection $f_0\in V_{\lambda_0}(\mbp_0)$ is called \emph{deformation rigid} if there exist a neighborhood $U\subset \MC_{2n}$ of $\mbp_0$ and a constant $\epsilon$ with the following property: for all $\mbp\in U$ if there exists a critical eigensection $f\in \oplus_{|\lambda-\lambda_0|<\ep}V_{\lambda}(\mbp)$, then there exists $R\in SO(3)$ such that $R\cdot \mbp=\mbp_0$ and $R\cdot f=f_0$.
%\end{definition}

Let $\mbp_0\in \MC_{2n}$ be a configuration with $2n\geq 4$ and let $f_0\in V_{\lambda_0}(\mbp_0)$ be a critical eigensection of unit $L^2$-norm.  
In what follows, we study the deformation rigidity of $f_0$ when $f_0$ is non-degenerate and minimal.  

\begin{remark}
	If $\mbp_0$ is the Taubes--Wu tetrahedral configuration and $V_{\lambda_0}(\mbp_0)$ is a critical eigenspace, then according to~\cite[Section~5]{chen2024existence} there exists a critical eigensection $f\in V_{\lambda_0}(\mbp_0)$ such that
	\[
	V_{\lambda_0}(\mbp_0)=\mathrm{span}\{f, L_1 f, L_2 f, L_3 f\}.
	\]
	Moreover, it is proved in \cite{chen2024existence} that every Taubes--Wu tetrahedral eigensection is non-degenerate and rigid.
\end{remark}

Given $\mbp_0\in\mathcal{C}_{2n}$, we choose an order on $\mbp_0$ as in the preceding section and write
\[
\mbp_0=(p_{1,0},\ldots,p_{2n,0}).
\]
Let $U_{\mbp_0}$ be a neighbourhood of $\mbp_0$ consisting of small disjoint geodesic balls centred at each $p_{j,0}$.  
For any $\mbp\in U_{\mbp_0}$, choose a diffeomorphism
\[
\varphi_{\mbp}: S^2\to S^2
\]
depending smoothly on $\mbp$, such that $\varphi_{\mbp}(\mbp_0)=\mbp$, $\vp_{\mbp}(U_{\mbp_0})\subset U_{\mbp_0}$ and $\vp_{\mbp_0}=\Id$. 
Notice that any $p\in U_{\mbp_0}$ inherits an order so that we can write  $\mbp=(p_1,\cdots,p_{2n})$, where $p_i$ lies in the geodesic ball centred at $p_{i,0}$.

Since $2n>2$, there exist two points—without loss of generality $p_1$ and $p_2$—such that $p_1\neq -p_2$. We can choose a unique rotation $R\in \SO(3)$ so that $R\cdot p_1=(0,0,1)$ 
(the north pole) and $R\cdot p_2\in\{y=0,\ x>0\}$.  
This yields in particular that the $\SO(3)$-action on $U_{\mbp_0}$ is free, and hence $\mathcal C_{2n}/\SO(3)$
is a smooth manifold near $[\mbp_0]$.

By Proposition~\ref{prop_deformation}, there exists an $\epsilon>0$ such that for each $\mbp\in U_{\mbp_0}$ the vector space
\[
V(\mbp)
:= \bigoplus_{|\lambda-\lambda_0|<\epsilon} V_{\lambda}(\mbp)
\]
is of dimension $4$. 
Moreover, this yields a trivial rank $4$ vector bundle over  $U_{\mbp_0}$.

For each $\mbp\in U_{\mbp_0}$, elements of $\varphi_{\mbp}^*V(\mbp)$ are sections of $\MI_{\mbp_0}$.  
Define the following linear functionals
\begin{equation}
	\chi_j : V(\mbp)\to\mathbb{R},\qquad
	\chi_j(f):=\langle L_j f_0,\ \varphi_{\mbp}^* f\rangle_{L^2}
\end{equation}
and set
\[
L := \bigcap_{j=1}^3 \ker\chi_j .
\]
Recalling that the minimality of $f_0$ implies linear independence of $L_1 f_0, L_2 f_0, L_3 f_0$, we obtain  that $\chi=(\chi_1, \chi_2,\chi_3)\colon V(p)\to \mathbb R^3$ has full rank at $\mbp_0$ and therefore also for all $\mbp$ sufficiently close to $\mbp_0$. 
We conclude that $L$ is a real line bundle over $U_{\mbp_0}$ such that $f_0\in L_{\mbp_0}$. 

Since $L$ inherits an Euclidean scalar product, there is a unique section $\mbp\mapsto f_{\mbp}\in L|_{\mbp}$ satisfying
\[
\|f_{\mbp}\|_{L^2}=1\qquad \text{and}\qquad f_{\mbp_0}=f_0
\]
so that the map 
\begin{equation}
	%\begin{split}
		\hat{\Theta}\colon U_{\mbp_0}\subset (S^2)^{2n} \longrightarrow \mathbb{C}^{2n},\qquad 
		\hat{\Theta}(\mbp):=\Tr(f_{\mbp})
	%\end{split}
\end{equation}
is well-defined. 
Write $\Tr(f_{\mbp})=(a_1,\ldots,a_{2n})$, where $a_i=a_i(\mbp)$ is the leading coefficient of $f_{\mbp}$ at $p_i$. Observe that the construction implies that $\hat\Theta$ is $\SO(3)$-invariant in the following sense: if $\mbp\in U_{\mbp_0}$ and $R\in SO(3)$ is such that $R\cdot\mbp\in U_{\mbp_0}$, then $\hat{\Theta}(R\cdot \mbp)=\hat{\Theta}(\mbp)$.

Let  $H_1$ denote the infinitesimal action of $\SO(3)$ at $\mbp_0$, i.e.,
\begin{equation}
	H_1:=\Span\left\{\frac{d}{dt}\Big|_{t=0}(g_t\cdot \mbp_0)\;:\; g_t\in SO(3)\right\}
	\subset T_{\mbp_0}(S^2)^{2n}.
\end{equation}  
%While the $\SO(3)$-action on $\mathcal C_{2n}$ is not free, the stabilizer of any point  $\mbp\in U_{\mbp_0}$ is trivial  as it was already observed above. Therefore, at least near $\mbp_0$ the quotient $\mathcal C_{2n}/\SO(3)$ is a manifold.  Slightly abusing notations, we write $U_{\mbp_0}/\SO(3)$ for a neighborhood of $[\mbp_0]\in \mathcal C_{2n}/\SO(3)$. 
Clearly, we have an isomorphism
\[
T_{\mbp_0}\big (U_{\mbp_0}/SO(3) \big )\cong T_{\mbp_0}\big (S^{2\,} \big)^{2n}/H_1.
\]
Also, define
\[
H_2(\mbp):=\Span\big \{\Tr(f)\mid f\in L_{\mbp}^{\perp}\subset V(\mbp) \big \}\subset \mathbb C^{2n},
\]
and note that for any $R\in SO(3)$ we have $H_2(R\cdot \mbp)=H_2(\mbp)$. In particular, we have
\begin{equation}
	\label{Eq_SubspH2}
	H_2(\mbp_0)
	=\Span\{\Tr(L_1 f_0),\,\Tr(L_2 f_0),\,\Tr(L_3 f_0)\}.
\end{equation}
A direct verification yields $\dim H_2(\mbp_0)=3$ and Proposition~\ref{prop_deformation} implies that $\dim H_2(\mbp)=3$ for all $\mbp$ sufficiently close to $\mbp_0$.
Thus, we obtain a real vector bundle $\MH$ over $U_{\mbp_0}/SO(3)$ defined by $\MH|_{[\mbp]}:=H_2(\mbp)$. 
By construction, $\MH$ is a subbundle of the product bundle $\ME:=U_{\mbp_0}/SO(3)\times \mathbb{C}^{2n}$ and we denote by $\MQ$ the quotient bundle:
\[
\MQ:=\ME/\MH.
\]
Observe that the product connection on $\ME$ induces a connection $\nabla$ on  $\MQ$.

%Consider the trivial bundle
%\[
%\ME:=U_{\mbp_0}/SO(3)\times \mathbb{C}^{2n},
%\]
%and let $\MH\subset \ME$ be the natural subbundle defined by
%\[
%\MH|_{[\mbp]}:=H_2(\mbp).
%\]
%We then define the quotient bundle
%
%Since $\ME$ is a trivial bundle, it carries a natural flat connection,
%and we equip $\MQ$ with the induced quotient connection, which we also
%denote by $\nabla$.

Next we define a section
\[
\Theta:U_{\mbp_0}/SO(3)\to \MQ\qquad\text{by setting}\qquad
\Theta([\mbp]):=[\hat{\Theta}(\mbp)].
\]
In particular, we have $\Theta([\mbp_0])=0$. Let
\begin{equation}
	\label{Eq_DerivOfTheta}
	\nabla\Theta_{[\mbp_0]}:\; T_{[\mbp_0]}(U_{\mbp_0}/SO(3))
	\longrightarrow \MQ_{[\mbp_0]} \cong \mathbb{C}^{2n}/H_2(\mbp_0)
\end{equation}
denote the covariant derivative of $\Theta$ at $\mbp_0$. 
%In what follows, we compute the covariant derivative $\nabla_v\Theta$ at $[\mbp_0]$ and prove that $\Theta$ is transverse to the zero section. 
%Since $\Theta([\mbp_0])=0$, the value of $\nabla_v\Theta$ at $[\mbp_0]$ does not depend on the choice of connection.

Let $\mbp_t=(p_{1,t},\ldots,p_{2n,t})$ be a unique geodesic connecting  $\mbp_0$ and $\mbp$. 
Denote by $\mbv$ its tangent vector, i.e.,
\[
\mbv := \frac{d}{dt}\Big|_{t=0} \mbp_t\in T_{\mbp_0}(S^2)^{2n}.
\]

\begin{proposition}
  \label{Prop_DerivTheta}
	The derivative of  $\Theta$ at $\mbp_0$ in the direction of $\mbv$ is 
	\begin{equation}
		\label{Eq_DerivOfThetaExpression}
	\nabla_\mbv\Theta_{[\mbp_0]} = \frac{3}{2}\, \mathbf{b}\cdot \mbv  \pmod{H_2 (\mbp_0)}.
	\end{equation}
\end{proposition}
\begin{proof}
We compute the derivative of $\hat \Theta$ in the direction of $\mbv$ first. Rewriting $\vp_{t}:=\vp_{\mbp_t}$ for convenience, we obtain a family of diffeomorphisms satisfying $\varphi_t(\mbp_0)=\mbp_t$. We thus wish to compute 
\[
\frac{d}{dt}\Big|_{t=0}\Tr(f_t), \qquad\text{where }\quad f_t := \varphi_t^*(f_{\mbp_t})\in\Gamma(\MI_{\mbp_0}).
\]
  
  Recall that $f_{\mbp_t}$ is a section of rank $4$ vector bundle $V\to U_{\mbp_0}$ spanned by eigensections and therefore can be written as $f_{\mbp_t} = \sum_{j=1}^4 f_{\mbp_t}^j$, where each component $f_{\mbp_t}^j$ satisfies 
  \begin{equation}
  	 \label{Eq_Aux_FamEigensections}
  	\Delta_{g_0} f_{\mbp_t}^j=\lambda_t^j f_{\mbp_t}^j.
  \end{equation}
 Without loss of generality, we can assume that $f_0^1=f_{\mbp_0} = f_0$ is the initial critical $\Z/2$ eigenfunction and $f_0^j=0$ for $j=2,3,4$.

Pulling~\eqref{Eq_Aux_FamEigensections} back by $\varphi_t$, we obtain 
\begin{equation}
	\label{Eq_FamEigensections}
\Delta_{\varphi_t^* g_0}(\varphi_t^* f_{\mbp_t}^j)=\lambda_t^j (\varphi_t^* f_{\mbp_t}^j)
\qquad\Longleftrightarrow\qquad
\Delta_{g_t} f_t^j = \lambda_t f_t^j,
\end{equation}
where we denote $g_t := \varphi_t^* g_0$ and $f_t^j:=\varphi_t^* f_{\mbp_t}^j$.
By \cite[Section~5.1]{donaldson2019deformations}, the derivative of the family $\Delta_t$ can be written in the following form: 
\[
\frac{d}{dt}\Big|_{t=0} \Delta_{g_t} h
= \nabla_{\mbv}(\Delta_{g_0} h) - \Delta_{g_0}(\nabla_{\mbv} h)
\qquad\text{for any section } h.
\]
Since $\lambda_t$ is differentiable at $t=0$ by Proposition~\ref{prop_deformation},  differentiating~\eqref{Eq_FamEigensections} at $t=0$ yields
\[
\Delta_{g_0}(\dot{f}_0^j - \nabla_{\mbv} f_0^j)
=
\lambda_0(\dot{f}_0^j - \nabla_{\mbv} f_0^j) + \dot{\lambda}_0^j\, f_0^j.
\]
Notice that the value of $\lambda_t^j$ at $t=0$ does not depend on $j$. 

Furthermore, near each singular point we have the expansion $f_0\sim |z|^{3/2}$, and hence
\[
\dot{f}_0-\nabla_{\mbv} f_0 \sim |z|^{1/2}.
\]
Integrating by parts, we obtain
\[
\int_{S^2} \langle (\Delta-\lambda_0)(\dot{f}_0-\nabla_{\mbv} f_0),\, f_0\rangle = \int_{S^2} \langle (\dot{f}_0-\nabla_{\mbv} f_0),\, (\Delta-\lambda_0) f_0\rangle =0.
\]
Since the left hand side of the above equality equals $\dot{\lambda}_0^1\| f_0\|^2_{L^2}$ and $\|f_0\|_{L^2}=1$, it follows that $\dot{\lambda}_0^1=0$. Therefore, we obtain that the equation
\[
\Delta(\dot{f}_0^j-\nabla_{\mbv} f_0^j)=\lambda_0(\dot{f}_0^j-\nabla_{\mbv} f_0^j)
\]
holds for any $j=1,2,3,4$, which yields $\dot{f_0}-\nabla_{\mbv} f_0\in V_{\lambda_0}$.  
Since $\Tr(h)\in H_2$ for all $h\in V_{\lambda_0}$, we conclude
\[
\Tr(\dot{f_0}-\nabla_{\mbv} f_0)\in H_2.
\]

Furthermore, if $z_j$ is the local coordinate centred at $p_{j,0}$, then the expansion of $f_0$ near $p_{j,0}$ takes the form
\[
f_0 \sim \Re\!\left(b_j\, z_j^{3/2}\right)+\MO\big (|z_j|^{5/2}\big ),
\]
so that we obtain a vector $\mathbf{b}=(b_1,\ldots,b_{2n})$ with all components being non-zero. 

For a tangent vector $\mathbf{v}=(v_1,\ldots,v_{2n})$, in a neighbourhood of $p_{j,0}$ we have
\[
\nabla_{\mathbf{v}} f_0
= \Re\!\left(\tfrac{3}{2}\, b_j\cdot v_j\, z_j^{1/2}\right)
+ \MO\big (|z_j|^{3/2}\big ).
\]
Therefore,
\[
\Tr(\nabla_{\mathbf{v}} f_0)=\tfrac{3}{2}\,\mathbf{b}\cdot\mathbf{v} := \tfrac 32 \big ( b_1\cdot v_1,\ldots, b_{2n}\cdot v_{2n} \big). 
\]
Consequently,
\[
\frac{d}{dt}\Big|_{t=0}\Tr(f_t)
= \Tr (\dot f_0)\equiv \Tr\big (  \nabla_\mbv f_0 \big )
=\tfrac{3}{2}\,\mathbf{b}\cdot \mbv
\quad\pmod{H_2}.
\]
We conclude that the covariant derivative of $\Theta$ is given by~\eqref{Eq_DerivOfThetaExpression} as claimed.
\end{proof}

\begin{corollary}\label{prop_isomorphism}
	$\Theta$ is transverse to the zero section.
\end{corollary}
\begin{proof}
  To prove that~\eqref{Eq_DerivOfTheta}  is an isomorphism, first observe that the linear map
  \[
  \Gamma: T_{\mbp_0}(S^2)^{2n} \longrightarrow \mathbb{C}^{2n},
  \qquad \Gamma(\mbv) := \frac{3}{2}\, \mathbf{b}\cdot \mbv
  \]
  is an isomorphism because $f_0$ is non-degenerate. Furthermore, the geometric origin
  of $H_1$ and $H_2$ implies that the variation of the trace along a
  rotational vector field corresponds exactly to the trace of the
  rotationally generated eigensection; explicitly, we have
  $\Gamma(H_1) = H_2$. 
  Therefore, $\Gamma$ descends to a well-defined isomorphism of the quotient
  spaces. 
  It remains to notice that this induced map between quotient spaces coincides with $\nabla\Theta_{[\mbp_0]}$ by Proposition~\ref{Prop_DerivTheta}.
\end{proof}

%\[
%V_{\mbp} = L_{\mbp} \oplus L_{\mbp}^{\perp},
%\qquad\text{and define}\qquad
%H_1(\mbp):=\Tr(L_{\mbp}),\quad H_2(\mbp):=\Tr(L_{\mbp}^{\perp}).
%\]
%Note that, since $L_{\mbp_0}^\perp$ is spanned by $L_1f_0, L_2f_0,
%L_3f_0$, we have
%\[
%H_1(\mbp_0)=\Span\{\Tr(f_0)\}=\{0\},\qquad\text{and}\qquad
%H_2(\mbp_0)=\Span\{\Tr(L_1 f_0),\,\Tr(L_2 f_0),\,\Tr(L_3 f_0)\}.
%\]
%Moreover, for $\mbp$ sufficiently close to $\mbp_0$, Proposition~\ref{prop_deformation} implies that $\dim H_2(\mbp)=3$.

We can now prove our main theorem.

%\begin{theorem}
%	A minimal non-degenerate critical $\Z/2$ eigenfunction $f_0$ is deformation rigid.
%\end{theorem}

\begin{proof}[Proof of Theorem~\ref{Thm_DeformRigid}]
  \label{Pf_MainThm}
	By Corollary~\ref{prop_isomorphism}, shrinking $U_{\mbp_0}$ if
	necessary, we may assume that the image of $\Theta([\mbp])$ in
	$\mathbb{C}^{2n}/H_2(\mbp)$ is non-zero for all
	$\mbp\in U_{\mbp_0}\setminus\{\mbp_0\}$.

	Furthermore, recall that we have the splitting $V ({\mbp}) = L_{\mbp} \oplus L_{\mbp}^{\perp}$ satisfying the following properties:
	\[
	\Ker \,\Tr|_{V(\mbp_0)} = L_{\mbp_0}=\mathbb R\, f_0\qquad\text{and}\qquad
	\Tr|_{L_{\mbp_0}^\perp}\colon  L_{\mbp_0}^\perp\to H_2(\mbp_0)\ \text{is an isomorphism}.
	\]
	Therefore, for any
	$f=f_1+f_2\in L_{\mbp}\oplus L_{\mbp}^{\perp}=V({\mbp})$, with
	$f_1\in L_{\mbp}$ and $f_2\in L_{\mbp}^{\perp}$, we  have
	\[
	\Tr(f_1)\neq 0 \mod{H_2(\mbp)}.
	\]
	Since $\Tr(f_2)\in H_2(\mbp)$, we conclude that $\Tr(f)\neq 0$.
	Thus no non-zero eigensection in $V({\mbp})$ has vanishing trace, and
	therefore $f_0$ is deformation rigid.
\end{proof}

\bibliographystyle{plain}
\bibliography{references}
\end{document}

%% file: convention.tex
\newcommand{\MH}{\mathcal{H}}

\newcommand{\MO}{\mathcal{O}}
\newcommand{\MQ}{\mathcal{Q}}

\newcommand{\cI}{{\mathcal I}}

\newcommand{\Ker}{\mathrm{Ker}}

\newcommand{\Z}{{\mathbb Z}}
\newcommand{\R}{{\mathbb R}}
\newcommand{\C}{{\mathbb C}}

\newcommand{\La}{\lambda}

\newcommand{\MC}{\mathcal{C}}
\newcommand{\Tr}{\mathrm{Tr}}

\newcommand{\vp}{\varphi}

\newcommand{\ME}{\mathcal{E}}

\newcommand{\lam}{\lambda}

\newcommand{\MI}{\mathcal{I}}
\newcommand{\ZT}{\mathbb{Z}/2}

\newcommand{\Id}{\mathrm{Id}}

\newcommand{\dist}{\mathrm{dist}}

\newcommand{\SO}{\mathrm{SO}}

\newcommand{\mbR}{\mathbb{R}}

\newcommand{\mbH}{\mathbb{H}}

\newcommand{\mul}{{\mathrm{mul}\,}}

\newcommand{\mbp}{\mathbf{p}}
\newcommand{\mbq}{\mathbf{q}}

\newcommand{\Span}{\mathrm{span}}
\newcommand{\mbv}{\mathbf{v}}

%% file: references.bib
@incollection {Donaldson17_AdiabaticLimits,
	AUTHOR = {Donaldson, Simon},
	TITLE = {Adiabatic limits of co-associative {K}ovalev-{L}efschetz
	fibrations},
	BOOKTITLE = {Algebra, geometry, and physics in the 21st century},
	SERIES = {Progr. Math.},
	VOLUME = {324},
	PAGES = {1--29},
	PUBLISHER = {Birkh\"auser/Springer, Cham},
	YEAR = {2017},
	ISBN = {978-3-319-59938-0; 978-3-319-59939-7},
	MRCLASS = {53C26 (53C38)},
	MRNUMBER = {3702382},
	MRREVIEWER = {Antonella\ Nannicini},
	DOI = {10.1007/978-3-319-59939-7\_1},
	URL = {https://doi.org/10.1007/978-3-319-59939-7_1},
}

@article {HE2023SpecialLagrangian,
    AUTHOR = {He, Siqi},
     TITLE = {The branched deformations of the special {L}agrangian
              submanifolds},
   JOURNAL = {Geom. Funct. Anal.},
  FJOURNAL = {Geometric and Functional Analysis},
    VOLUME = {33},
      YEAR = {2023},
    NUMBER = {5},
     PAGES = {1266--1321},
      ISSN = {1016-443X,1420-8970},
  MRNUMBER = {4646409},
       DOI = {10.1007/s00039-023-00645-8},
       URL = {https://doi.org/10.1007/s00039-023-00645-8},
}

@article {HMT2023IndexGraphs,
    AUTHOR = {Haydys, Andriy and Mazzeo, Rafe and Takahashi, Ryosuke},
     TITLE = {An index theorem for {$\mathbb Z/2$}-harmonic spinors
              branching along a graph},
   JOURNAL = {arXiv preprint arXiv:2310.15295},
      YEAR = {2023},
    EPRINT = {2310.15295},
ARCHIVEPREFIX = {arXiv},
PRIMARYCLASS = {math.DG},
       DOI = {10.48550/arXiv.2310.15295},
       URL = {https://arxiv.org/abs/2310.15295},
      NOTE = {Version 2, 2025},
}

@article {Parker2025DeformationsZ2Spinors,
    AUTHOR = {Parker, Gregory J.},
     TITLE = {Deformations of {$\mathbb Z_2$}-harmonic spinors on
              3-manifolds},
   JOURNAL = {Geom. Funct. Anal.},
  FJOURNAL = {Geometric and Functional Analysis},
      YEAR = {2026},
      ISSN = {1016-443X,1420-8970},
       DOI = {10.1007/s00039-026-00729-1},
       URL = {https://doi.org/10.1007/s00039-026-00729-1},
    EPRINT = {2301.06245},
ARCHIVEPREFIX = {arXiv},
PRIMARYCLASS = {math.DG},
}

@article {HeParker2024Z2Harmonic,
    AUTHOR = {He, Siqi and Parker, Gregory J.},
     TITLE = {{$\mathbb Z_2$}-harmonic spinors and 1-forms on connected
              sums and torus sums of 3-manifolds},
   JOURNAL = {arXiv preprint arXiv:2407.10922},
      YEAR = {2024},
    EPRINT = {2407.10922},
ARCHIVEPREFIX = {arXiv},
PRIMARYCLASS = {math.DG},
       DOI = {10.48550/arXiv.2407.10922},
       URL = {https://arxiv.org/abs/2407.10922},
}

@article {HaydysSalm2025Z2Sphere,
    AUTHOR = {Haydys, Andriy and Salm, Willem Adriaan},
     TITLE = {Search for {$\mathbb Z/2$} eigenfunctions on the sphere
              using machine learning},
   JOURNAL = {arXiv preprint arXiv:2507.13122},
      YEAR = {2025},
    EPRINT = {2507.13122},
ARCHIVEPREFIX = {arXiv},
PRIMARYCLASS = {math.DG},
       DOI = {10.48550/arXiv.2507.13122},
       URL = {https://arxiv.org/abs/2507.13122},
}

@article {FranceschiniMazzeo2026MinimalSurfaces,
    AUTHOR = {Franceschini, Federico and Mazzeo, Rafe and Minter, Paul},
     TITLE = {Minimal surfaces with stratified branching sets},
   JOURNAL = {arXiv preprint arXiv:2603.27168},
      YEAR = {2026},
    EPRINT = {2603.27168},
ARCHIVEPREFIX = {arXiv},
PRIMARYCLASS = {math.DG},
       DOI = {10.48550/arXiv.2603.27168},
       URL = {https://arxiv.org/abs/2603.27168},
}

@article {chen2024existence,
    AUTHOR = {Chen, Jiahuang and He, Siqi},
     TITLE = {On the existence and rigidity of critical {$\mathbb Z/2$}
              eigenvalues},
   JOURNAL = {arXiv preprint arXiv:2404.05387},
      YEAR = {2024},
    EPRINT = {2404.05387},
ARCHIVEPREFIX = {arXiv},
PRIMARYCLASS = {math.DG},
       DOI = {10.48550/arXiv.2404.05387},
       URL = {https://arxiv.org/abs/2404.05387},
}

@article {taubes2021topological,
    AUTHOR = {Taubes, Clifford H. and Wu, Yingying},
     TITLE = {Topological aspects of {$\mathbb Z/2\mathbb Z$}
              eigenfunctions for the {L}aplacian on {$S^2$}},
   JOURNAL = {J. Differential Geom.},
  FJOURNAL = {Journal of Differential Geometry},
    VOLUME = {128},
      YEAR = {2024},
     PAGES = {379--462},
      ISSN = {0022-040X,1945-743X},
       URL = {https://arxiv.org/abs/2108.05017},
    EPRINT = {2108.05017},
ARCHIVEPREFIX = {arXiv},
PRIMARYCLASS = {math.DG},
}

@incollection {taubes2020examples,
    AUTHOR = {Taubes, Clifford H. and Wu, Yingying},
     TITLE = {Examples of singularity models for {$\mathbb Z/2$} harmonic
              1-forms and spinors in dimension three},
 BOOKTITLE = {Proceedings of the {G}{\"o}kova Geometry-Topology Conferences
              2018/2019},
     PAGES = {37--66},
 PUBLISHER = {International Press, Somerville, MA},
      YEAR = {2020},
  MRNUMBER = {4251085},
       URL = {https://gokovagt.org/proceedings/2018-2019/03-TaubWu.html},
    EPRINT = {2001.00227},
ARCHIVEPREFIX = {arXiv},
PRIMARYCLASS = {math.DG},
}

@article {taubes2013psl2c,
    AUTHOR = {Taubes, Clifford H.},
     TITLE = {{${\rm PSL}(2;\mathbb C)$} connections on 3-manifolds with
              {${\rm L}^2$} bounds on curvature},
   JOURNAL = {Camb. J. Math.},
  FJOURNAL = {Cambridge Journal of Mathematics},
    VOLUME = {1},
      YEAR = {2013},
    NUMBER = {2},
     PAGES = {239--397},
      ISSN = {2168-0930,2168-0949},
   MRCLASS = {53B15 (53C05 53C21)},
  MRNUMBER = {3272050},
       DOI = {10.4310/CJM.2013.v1.n2.a2},
       URL = {https://doi.org/10.4310/CJM.2013.v1.n2.a2},
      NOTE = {Corrigendum: Camb. J. Math. 3 (2015), no. 4, 619--631},
}

@article {taubes2014zero,
    AUTHOR = {Taubes, Clifford H.},
     TITLE = {The zero loci of {$\mathbb Z/2$} harmonic spinors in dimension
              2, 3 and 4},
   JOURNAL = {arXiv preprint arXiv:1407.6206},
      YEAR = {2014},
    EPRINT = {1407.6206},
ARCHIVEPREFIX = {arXiv},
PRIMARYCLASS = {math.DG},
       DOI = {10.48550/arXiv.1407.6206},
       URL = {https://arxiv.org/abs/1407.6206},
}

@article {donaldson2019deformations,
    AUTHOR = {Donaldson, Simon},
     TITLE = {Deformations of multivalued harmonic functions},
   JOURNAL = {Q. J. Math.},
  FJOURNAL = {The Quarterly Journal of Mathematics},
    VOLUME = {72},
      YEAR = {2021},
    NUMBER = {1-2},
     PAGES = {199--235},
      ISSN = {0033-5606,1464-3847},
  MRNUMBER = {4271385},
       DOI = {10.1093/qmath/haab018},
       URL = {https://doi.org/10.1093/qmath/haab018},
}

@article {walpuskizhang2019compactness,
    AUTHOR = {Walpuski, Thomas and Zhang, Boyu},
     TITLE = {On the compactness problem for a family of generalized
              {S}eiberg--{W}itten equations in dimension three},
   JOURNAL = {Duke Math. J.},
  FJOURNAL = {Duke Mathematical Journal},
    VOLUME = {170},
      YEAR = {2021},
    NUMBER = {17},
     PAGES = {3891--3934},
      ISSN = {0012-7094,1547-7398},
       DOI = {10.1215/00127094-2021-0005},
       URL = {https://doi.org/10.1215/00127094-2021-0005},
    EPRINT = {1904.03749},
ARCHIVEPREFIX = {arXiv},
PRIMARYCLASS = {math.DG},
}

@article {haydyswalpuski2015compactness,
    AUTHOR = {Haydys, Andriy and Walpuski, Thomas},
     TITLE = {A compactness theorem for the {S}eiberg--{W}itten equation
              with multiple spinors in dimension three},
   JOURNAL = {Geom. Funct. Anal.},
  FJOURNAL = {Geometric and Functional Analysis},
    VOLUME = {25},
      YEAR = {2015},
    NUMBER = {6},
     PAGES = {1799--1821},
      ISSN = {1016-443X,1420-8970},
  MRNUMBER = {3432158},
       DOI = {10.1007/s00039-015-0346-3},
       URL = {https://doi.org/10.1007/s00039-015-0346-3},
}

@article {taubes2013compactness,
    AUTHOR = {Taubes, Clifford H.},
     TITLE = {Compactness theorems for {${\rm SL}(2;\mathbb C)$}
              generalizations of the 4-dimensional anti-self dual equations},
   JOURNAL = {arXiv preprint arXiv:1307.6447},
      YEAR = {2013},
    EPRINT = {1307.6447},
ARCHIVEPREFIX = {arXiv},
PRIMARYCLASS = {math.DG},
       DOI = {10.48550/arXiv.1307.6447},
       URL = {https://arxiv.org/abs/1307.6447},
      NOTE = {Version 5, 2020},
}
